\documentclass{article} %[<release date>] envcountsame,

\usepackage{amssymb} %amsmath,amsthm,
\usepackage{amsthm}
\usepackage[fleqn]{amsmath}
\usepackage{bbm} %Indicator 1
\usepackage{graphicx}
\usepackage{float} %for figure placement specification ``H''
\usepackage{color}
\usepackage{hyperref} % [hypertexnames=false] for clickabele refs, Attention: delete old .toc file before compiling with hyperref package
\usepackage[retainorgcmds]{IEEEtrantools} % for too large equations

\usepackage{etoolbox} %maces Hyperref and IEEEeqnarray go well together
\AtBeginEnvironment{IEEEeqnarray}{\phantomsection}
\AtBeginEnvironment{IEEEeqnarray*}{\phantomsection}
\newcommand{\RHT}{R^T_H}
\newcommand{\RHTts}{R^T_H(t,s)}
\newcommand{\RHts}{R_H(t,s)}
\newcommand{\RH}{R_H}

\newcommand{\Vtth}{Var(X^H_{t+h}-X^H_t)}
\newcommand{\vth}{\sigma_{t,h}^2}

\newcommand{\btso}{\beta_{t,s}^1}
\newcommand{\btst}{\beta_{t,s}^2}
\newcommand{\btho}{\beta_{t,t+h}^1}
\newcommand{\btht}{\beta_{t,t+h}^2}
\newcommand{\bsto}{\beta_{s,t}^1}
\newcommand{\bstt}{\beta_{s,t}^2}
\newcommand{\bhto}{\beta_{t+h,t}^1}
\newcommand{\bhtt}{\beta_{t+h,t}^2}

\newtheorem{theorem}{Theorem}[subsection]%[level of numbering]
\newtheorem{proposition}[theorem]{Proposition}%{calls it}[numbering same as]{what will be written}
\newtheorem{lemma}[theorem]{Lemma}
\newtheorem{corollary}[theorem]{Corollary}
\newtheorem{definition}[theorem]{Definition}
\theoremstyle{remark}%from here one the \newtheorems will look diffrent
\newtheorem{remark}[theorem]{Remark}

\title{Fractional Brownian Motion with Variable Hurst Parameter: Definition and Properties}
\author{Jelena Ryvkina\\
Department of Mathematics, Tufts University,\\ 
\href{mailto:jelena.ryvkina@tufts.edu}{jelena.ryvkina@tufts.edu}}
\begin{document}

\maketitle

\begin{abstract} A class of Gaussian processes generalizing the usual \emph{fractional Brownian motion} for Hurst indices in (1/2,1) and \emph{multifractal Brownian motion} introduced in \cite{multifrac} and \cite{bouf} is presented. Any measurable function assuming values in this interval can now be chosen as a variable Hurst parameter. These processes allow for modeling of phenomena where the regularity properties can change with time either continuously or through jumps, such as in the volatility of a stock or in Internet traffic. Some properties of the sample paths of the new process class, including different types of continuity and long-range dependence, are discussed. It is found that the regularity properties of the Hurst function chosen directly correspond to the regularity properties of the sample paths of the processes. The long-range dependence property of fractional Brownian motion is preserved in the larger process class. As an application, Fokker-Planck-type equations for a time-changed \emph{fractional Brownian motion with variable Hurst parameter} are found. 
\end{abstract}
%\keywords{Fractional Brownian motion \and Gaussian processes \and Variable Hurst parameter \and Self-similarity \and Sample path regularity}
%\subclass{60G15 \and 60G17 \and 60G22}



\section{Introduction}

\emph{Fractional Brownian motion (fBM)} is a class of zero-mean Gaussian processes with covariance functions given by

$$Cov(t,s)=\frac{1}{2}\left(s^{2H}+t^{2H}-|t-s|^{2H}\right)\mbox{ for }s,t\geq0,$$
where the so-called Hurst parameter $H$ is a constant in (0,1). This process was first introduced by Kolmogorov in 1940 \cite{Kolm}. The name \emph{fBM} goes back to Mandelbrot and van Ness, who found a representation as an integral with respect to regular \emph{Brownian motion (BM)} in 1968 \cite{MandNess}. \emph{fBM} has stationary increments, is a.s. H\"older continuous of any order $\gamma<H$ on compact intervals, and is $H$-self-similar.\\

\emph{fBM} is used for models in numerous areas such as finance, hydrology, and telecommunications. The Hurst parameter is what determines the path behavior of \emph{fBM} and hence the value of $H$ directly reflects the types of phenomena which can be modeled by the corresponding \emph{fBM}. In particular, the case $H<1/2$ yields negatively correlated increments and rather erratic paths suitable for modeling systems with fast changing states. In the case $H>1/2$, the increments are positively correlated and the process exhibits long-range dependence, i.e., \emph{fBM} with such $H$ is used when slowly decaying effects are observed. The case $H=1/2$ yields the usual \emph{BM} and hence independent increments are obtained \cite{ochs}.\\

In nature, one often observes changes in the dynamics of a system over time. For example, during market shocks, one can deduce unusual movements in the volatility of a stock price. Even though, due to long-range dependence in the data, \emph{fBM} with $H>1/2$ seems to be a good choice for modeling the volatility - as was suggested in \cite{Comte}, it was demonstrated in \cite{Chron} that different $H$ values for different time intervals are necessary.  Another example is the use of \emph{fBM} in image processing. In \cite{Jenn}, projections of images of osteoporosis affected bone matter are modeled as \emph{fBM} with a spatial index. It is likely that the model could be improved, where the Hurst parameter allowed to vary over the different areas of the sample, since the disease does not affect all bone matter uniformly. Further examples, where a constant $H$ seems unrealistic, can be easily found in computer traffic \cite{Vehel} and other fields. Consequently, a Gaussian  process exhibiting long-range dependence and variable in its path behavior is of interest not only theoretically but for various applications as well.\\

The objective of this paper is to define a class of Gaussian processes, extending \emph{fBM} for $H>1/2$ by permitting arbitrary measurable functions $H(\cdot)$ as variable Hurst parameters. Some initial properties such as path-regularity properties and long-range dependence shall be derived. Other generalizations of \emph{fBM} have been considered before, but the class of permissible Hurst index functions was usually restricted due to the purpose or the inherent properties of the respective construction. \emph{Multifractal Brownian motion (mfBM)} for example was introduced in \cite{multifrac} and \cite{bouf} and involves H\"older continuous parameter functions. For the process class considered in \cite{PeltVeh} and \cite{BenJaff}, the aim was to obtain sufficiently regular processes and hence only H\"older continuous parameter functions were used in the definition of \emph{multifractional Brownian motion (mBM)} as well. A further extension of \emph{mBM} was then presented in \cite{GMBM} and includes some lower semi-continuous functions, but a construction involving limits of H\"older functions and limits of corresponding processes was needed in order to maintain the regularity properties of the resulting processes. Other generalizations (see, e.g., \cite{pfBM}) are limited to piecewise constant $H$.\\

Here, a different approach inspired by Decreusefond's idea in \cite{Decreus} involving an appropriate class of covariance functions is pursued. His focus on Volterra processes yields a generalization of \emph{fBM} under some strong regularity assumption on the function $H(\cdot)$. However, in Section 2 of this paper, it will be shown that it is possible to extend this approach. The use of covariance functions to define Gaussian processes presents a coherent way to extend \emph{fBM} and \emph{mfBM} to a class of processes parametrized by the set of all measurable functions with values in (1/2,1) and different from \emph{mBM}. Possible discontinuities in the paths of the resulting processes are of interest for applications involving non-continuous but long-range dependent processes, especially since the location of their occurrence can be controlled through the parameter function, as will be seen in Section 3.2 . \\  

This paper is structured as follows: In Section 2, the class of covariance functions $\RH$ is defined and several representations of $\RH(\cdot,\cdot)$ are given. \emph{Fractional Brownian motion with variable Hurst parameter (fBMvH)} is defined. In Section 3, asymptotic behavior of the covariance functions is used to investigate some path properties of \emph{fBMvH}. Interestingly, the regularity properties of the Hurst function turn out to be directly reflected in the regularity properties of the paths of \emph{fBMvH} (e.g., jumps yield jumps, H\"older continuity yields H\"older continuity). Further, it is shown that the long-range dependence property, which makes \emph{fBM} with $H>1/2$ a realistic model for persistent systems, is preserved for the extended class. As an application, in Section 4, a Fokker-Plank-type equation for the time-changed \emph{fBMvH} is presented. In Section 5, the process is modeled and plots of several different \emph{fBMvHs} are shown.

\section{\textit{fBMvH}}

The first objective in obtaining an extension of the \emph{fBM} class will be to find a class of covariance functions involving a time varying $H$. Once a suitable class of functions is found, it is straightforward to define a new class of Gaussian processes.\\

In the case $H>1/2$, the covariance function of \emph{fBM} can be obtained as
    \begin{eqnarray}
    Cov(t,s)&=&\int_0^{s\wedge t}K(t,u)K(s,u)\,du,\mbox{ where}\nonumber\bigskip\\
    K(t,s)&=&\left[c_{H}s^{\frac{1}{2}-H}\int_s^t(u-s)^{H-\frac{3}{2}}u^{H-\frac{1}{2}}\,du\right]\mathbbm{1}_{[0,t]}(s),\label{eq0}
    \end{eqnarray}
for $c_{H}=\left(\frac{H(2H-1)}{\beta(2-2H,H-1/2)}\right)^{1/2}$, where $\beta(\cdot,\cdot)$ denotes the Beta function (see e.g., \cite{Nua2}). This paper makes use of L. Decreusefond's idea in \cite{Decreus}, where under certain regularity assumptions, Volterra processes were obtained by replacing the constant $H$ in (\ref{eq0}) with a function $H(t)$. It will be shown however that the measurability of $H(\cdot)$ alone is sufficient in order to define a Gaussian process.\\

Let $T>0$ and $H:[0,T]\rightarrow(1/2,1)$ be a function and $R_H^T:[0,T]\times[0,T]\rightarrow\mathbb{R}$ be defined by
\begin{eqnarray}\label{cov}
\RHTts & = & \int_0^{s\wedge t}K_H(t,u)K_H(s,u)\,du,
\end{eqnarray}
with 
\begin{eqnarray}\label{eq1}
K_H(t,s)& = & \left[c_{H_t}s^{\frac{1}{2}-H_t}\int_s^t(u-s)^{H_t-\frac{3}{2}}u^{H_t-\frac{1}{2}}\,du\right]\mathbbm{1}_{[0,t]}(s)
\end{eqnarray}
for $s,t\in[0,T]$ and where $H_t=H(t).$ Clearly, $\RHT$ is symmetric. Furthermore, $\RHT$ can be used as covariance function for a Gaussian process on $[0,T]$ due to the next theorem.
\begin{theorem}\label{psd} Let $H$ be measurable. Then, $\RHT$ is positive semidefinite.
\end{theorem}
\begin{proof} Let $K$ be the integral operator defined on $L^2([0,T])$ via $K_H$, i.e., $Kf(t)=\int_0^{T}K_H(t,s)f(s)\,ds.$ If $K_H$ is in $L^2([0,T]\times[0,T])$, then  $K$ has an adjoint on  $L^2([0,T])$ (see \cite{Con}) given by $K^*g(s)=\int_s^TK_H(t,s)g(t)\,dt$ for $g\in L^2.$ For $\alpha_1,...,\alpha_n\in\mathbb{R}$ and $t_1,...,t_n\in[0,T]$,
\begin{eqnarray}\label{eq2}
\sum\limits_{i,j}\alpha_i\alpha_jR_H^T(t_i,t_j)=\int_0^T[K^*(\sum\limits_i\alpha_i\delta_{t_i})(s)]^2\,ds\geq0,
\end{eqnarray}
where $\delta_t$ denotes the Dirac delta. It remains to show that $\int_{[0,T]^2}K_H(t,s)^2\,ds\,dt$ is finite.
\begin{IEEEeqnarray*}{lLl}
\int_{[0,T]^2}K_H(t,s)^2\,ds\,dt&&\bigskip\\ 
=\int_0^T\int_0^T\int_{u\vee v}^Tc_{H_t}^2(uv)^{H_t-\frac{1}{2}}\int_0^{u\wedge v}s^{1-2H_t}(u-s)^{H_t-\frac{3}{2}}(v-s)^{H_t-\frac{3}{2}}\,ds\,dt\,du\,dv.&&\bigskip\\
=\int_0^T\int_0^T\int_{u\vee v}^Tc_{H_t}^2(u-v)^{2H_t-2}\beta(H_t-1/2,2-2H_t)\,dt\,du\,dv\mbox{ by Lemma }\ref{AppB1}&&\bigskip\\ 
=\int_0^Tc_{H_t}^2\beta(H_t-1/2,2-2H_t)\int_0^t\int_0^t(u-v)^{2H_t-2}\,du\,dv\,dt&&\bigskip\\ 
=\int_0^Tt^{2H_t}\,dt,&&\\
\end{IEEEeqnarray*}
which is finite for all measurable $H$ with values in (1/2,1). %\qed
\end{proof}
The following theorem shows that (\ref{eq2}) is in fact strict and will become useful in Section 5.
\begin{theorem}\label{pd} Let $H$ be measurable. Then, $\RHT$ is positive definite.
\end{theorem}
\begin{proof} Without loss of generality, let $n\geq2$ be fixed and $t_1,...,t_n\in[0,T]$ be distinct. Further, let  $\alpha_1,...,\alpha_n\in\mathbb{R}-\{0\}$. By (\ref{eq2}) it is sufficient to show that $K^*(\sum\limits_i\alpha_i\delta_{t_i})(s)=\sum\limits_i\alpha_iK_H(t_i,s)\neq0$ on some interval contained in $[0,T].$ Assume the $t_i$ are ordered by magnitude, i.e., $t_1<t_2<...<t_n.$ For $s\in(t_{n-1},t_n)$ 
\begin{eqnarray*}\sum\limits_i\alpha_iK_H(t_i,s)=\alpha_nK_H(t_n,s)\neq 0
\end{eqnarray*}
by (\ref{eq1}) and hence
\begin{eqnarray*}
\sum\limits_{i,j}\alpha_i\alpha_jR_H^T(t_i,t_j)\geq\alpha_n^2\int_{t_{n-1}}^{t_n}K_H(t_n,s)^2\,ds>0.
\end{eqnarray*}
%\qed
\end{proof}

$\RHT$ is easily extended to the positive real line. Let $H:[0,\infty)\rightarrow(1/2,1)$ be measurable and let the covariance function $\RH$ be defined by $R_H:[0,\infty)\times[0,\infty)\rightarrow\mathbb{R}_+$ with $R_H(t,s)=R_H^T(t,s)$ for $t,s\leq T$. This definition is consistent, since $R_H^{T_1}(t,s)=R_H^{T_2}(t,s)$ on $[0,T_1\wedge T_2]\times[0,T_1\wedge T_2]$. $R_H$ is positive semidefinite (and positive), because $R_H^T$ has this property for all $T$.

\begin{definition}\label{dXH} \emph{Fractional Brownian motion with variable Hurst parameter (fBMvH)} is defined as the centered Gaussian process $X^H$ on $[0,\infty)$ starting at zero and with covariance function given by $R_H$. 
\end{definition}

For constant $H$, $X^H$ coincides with the usual \emph{fBM} on $\mathbb{R}_+$ and hence the class of processes $\{X^H:\,H:[0,\infty)\rightarrow(1/2,1)\mbox{ is measurable}\}$ is a generalization of \emph{fBM} with $H>1/2$. In the proof of Theorem \ref{psd} it was shown that $K_H(\cdot,\cdot)\in L^2([0,T]\times[0,T])$ and hence $K_H(t,\cdot)\in L^2([0,T])$ for each $t\in(0,T]$ fixed. By the It\^o isometry, $X^H$ can be constructed on $[0,T]$ as 

\begin{IEEEeqnarray}{lLl}\label{BMrep}
X^H_t=\int\limits_{0}^{t}K_H(t,s)\,dB_s\mbox{ for } t\in[0,T],
\end{IEEEeqnarray}
where $(B_s)_{s\in[0,T]}$ is a regular \emph{BM}. For H\"older continuous parameter functions $H$, \emph{fBMvH} coincides with \emph{mfBM} studied in \cite{multifrac} and \cite{bouf}.\\

\begin{remark} Under further regularity assumptions on the function $H,$ i.e., if $H$ assumes values in a compact interval and if the H\"older exponent $\alpha$ of $H$ satisfies $\alpha>\sup_tH_t,$ \emph{mfBM} is shown to be locally asymptotically self-similar in \cite{bouf}. Further, if $H$ is additionally assumed to be of bounded variations, the authors of \cite{bouf} show that the stochastic calculus developed in \cite{Nua2} holds for \emph{mfBM} and establish results on the local time of the process.  
\end{remark}

Several forms of $\RH$ are useful and will be given in the next proposition. One of them involves the Gauss hypergeometric function defined by
\begin{IEEEeqnarray}{lCr}\label{2F1}
\,_2F_1(a,b;c;z)&=&\sum\limits_{n=0}^{\infty}\frac{(a)_n(b)_n}{(c)_n}\frac{z^n}{n!}\mbox{ for }|z|<1,
\end{IEEEeqnarray}
on the unit disc and analytically extendible. $(x)_n$ denotes here the Pochhammer symbol, i.e., $(x)_0=1$ and $(x)_n=x(x+1)\cdots(x+n-1)$ for $n>0$. For properties of $\,_2F_1$ the interested reader is referred to Chapter 15 of \cite{Abram}.\\

In the following, let
\begin{IEEEeqnarray}{lCr}\label{cofs1}
c_{t,s}=c_{H_t}c_{H_s}&\mbox{ and }&\tilde{\beta}_{t,s}=\btso\btst+\bsto\bstt,
\end{IEEEeqnarray} 
both of which are symmetric in $s$ and $t$, with
\begin{IEEEeqnarray}{lCr}\label{cofs2}
\btso=\beta(H_t-1/2,2-H_t-H_s),&\ &\btst=\beta(H_t-H_s+1,H_t+H_s-1).\IEEEeqnarraynumspace
\end{IEEEeqnarray}

\begin{proposition}[Various forms of $\RH$] The following are equivalent forms of $\RH.$  For $s,t\geq0$\medskip\\
i)
\begin{IEEEeqnarray}{lLl}\label{form1}
&&\RHts=c_{t,s}\Bigg[\tilde{\beta}_{t,s}\frac{(s\wedge t)^{H_t+H_s}}{H_t+H_s}+\beta_{s\wedge t,s\vee t}^1\int_{s\wedge t}^{s\vee t}\int_0^{{s\wedge t}}\left(\frac{y}{z}\right)^{\alpha}(y-z)^{H_t+H_s-2}\,dz\,dy\Bigg],
\end{IEEEeqnarray}
where $\alpha=(H_t-H_s)\mathbbm{1}_{\{s\leq t\}}+(H_s-H_t)\mathbbm{1}_{\{s> t\}}$.\\

\noindent Assuming $t\geq s\geq 0$, (\ref{form1}) reduces to
\begin{IEEEeqnarray}{lLl} \label{form1.1}
&&R_H(t,s)=c_{t,s}\left[\tilde{\beta}_{t,s}\frac{s^{H_t+H_s}}{H_t+H_s}+\bsto\int_s^t\int_0^s\left(\frac{y}{z}\right)^{H_t-H_s}(y-z)^{H_t+H_s-2}\,dz\,dy\right].%\IEEEeqnarraynumspace 
\end{IEEEeqnarray}
\noindent ii)
\begin{IEEEeqnarray}{lLl}\label{form2} %
&&\RHts=\frac{c_{t,s}}{H_t+H_s}\Bigg[\btso\btst s^{H_t+H_s}+\bsto\Bigl[s^{H_t+H_s}\int_{\frac{s}{t}}^1\nu^{-2H_t}(1-\nu)^{H_s+H_t-2}\,d\nu\nonumber\bigskip\\
&&\ \ \ \ \ \ \ \ \ \ \ \ \ \ \ \ +\,t^{H_t+H_s}\int_0^{\frac{s}{t}}\nu^{H_s-H_t}(1-\nu)^{H_s+H_t-2}\,d\nu\Bigl]\Bigg], \IEEEeqnarraynumspace
\end{IEEEeqnarray}
where by convention, $0\cdot\infty=0$ and hence $R_H^T(t,0) =0;$\\
iii)
\begin{IEEEeqnarray}{lLl}\label{form3}
&&\RHts\medskip\\
&&\ \ \ =\frac{c_{t,s}}{H_t+H_s}\Bigg[\btso\btst s^{H_t+H_s}+\bsto\bstt t^{H_t+H_s}+\frac{\beta_{s,t}^1t(t-s)^{H_t+H_s-1}}{H_t+H_s-1}\left(\frac{t}{s}\right)^{H_t-H_s}\nonumber\bigskip\\
&&\ \ \ \ \ \ \ \ \ \times[\,_2F_1(1,2H_t;H_s+H_t;\frac{s-t}{s})-\,_2F_1(1,H_t-H_s;H_s+H_t;\frac{s-t}{s})]\Bigg].\nonumber
\end{IEEEeqnarray}
iv) 
\begin{IEEEeqnarray}{lLl}\label{form4} %\IEEEeqnarraynumspace
&&\RHts=c_{t,s}\int_{0}^t\int_0^s\Big[\bsto\mathbbm{1}_{\{y>z\}}+\btso\mathbbm{1}_{\{y<z\}}\Big]\left(\frac{y}{z}\right)^{H_t-H_s}|y-z|^{H_s+H_t-2}\,dz\,dy.
\end{IEEEeqnarray}
v) For $t\geq s\geq0$
\begin{IEEEeqnarray}{lLl}\label{form5} %\IEEEeqnarraynumspace
&&\RHts=\frac{c_{t,s}}{H_t+H_s}\Bigg[\frac{\btso\btst}{H_t+H_s} s^{H_t+H_s}+\frac{\bsto\bstt}{H_t+H_s} t^{H_t+H_s}\bigskip\\
&&\ \ \ \ \ \ \ \ \ \ \ \ \ \ \ \ \ -\bsto\int_{s}^t\int_z^t\left(\frac{y}{z}\right)^{H_t-H_s}(y-z)^{H_s+H_t-2}\,dy\,dz\Bigg].\nonumber
\end{IEEEeqnarray}

\end{proposition}

\begin{proof} Form (\ref{form1.1}), (\ref{form1}) and (\ref{form2}) are derived using Lemma \ref{AppB1}. The integrals in representation (\ref{form2}) can be expressed in terms of the Gauss hypergeometric function (see \cite{Prud}, p. 301 or Lemma \ref{Abramo} $(v)$, in Appendix B). Using 15.3.9 in \cite{Abram} (Lemma \ref{Abramo} $(i)$, in Appendix B), form (\ref{form3}) is obtained. (\ref{form4}) and (\ref{form5}) were derived for \emph{mfBM} in \cite{bouf} and \cite{multifrac} respectively. %\qed
\end{proof}

\begin{corollary}\label{cfBM} For $s,t\in[0,T]$ such that $H_t=H_s$
$$\RHTts=\frac{1}{2}(s^{2H_t}+t^{2H_t}-|t-s|^{2H_t}).$$
Further, the variance function of fBMvH $X^H$ is given by $R_H^T(t,t)=t^{2H_t}.$
\end{corollary}

\begin{remark} For H\"older continuous parameter functions $H$, \emph{mBM} was introduced in \cite{BenJaff} and \cite{PeltVeh} and has a covariance function of type $g(H_t,H_s)(t^{H_t+H_s}+s^{H_t+H_s}+|t-s|^{H_t+H_s})$ (see \cite{Cohen}), where $g(x,y)$ is smooth on $(\frac{1}{2},1)\times(\frac{1}{2},1)$ and symmetric in $x$ and $y$, and  $g(H_t,H_s)$ does not directly depend on $s$ and $t$. Hence, for general H\"older continuous $H$, the covariance functions of \emph{fBMvH} and \emph{mBM} cannot be transformed into each other upon multiplication by a function solely dependent on $H_s$ and $H_t$ (see, e.g., form (\ref{form3})). Accordingly, the processes cannot be generally obtained from each other through just $H(\cdot)$-dependent normalization, i.e., for a general H\"older continuous $H$, \emph{mBM} and \emph{fBMvH} will exhibit different dependence structures.
\end{remark}

\section{Properties}

\subsection{(Non)-self-similarity}

In order to achieve variability in the path behavior over time for the \emph{fBMvH} class, clearly, the stationarity of the increments as well as self-similarity properties had to be relaxed. Instead, the following proposition holds.

\begin{proposition}\label{gss} For any $a>0$ fixed, fBMvH $X^H$ satisfies
$$X^H_{at}\sim a^{\overline{H}_t}X^{\overline{H}}_t$$
in the sense of finite dimensional distributions and with $\overline{H}(t)=H(at).$
\end{proposition}

\begin{proof} Using (\ref{cov}), one can easily check that $Cov(X^H_{at},X^H_{as})=a^{\overline{H}_t+\overline{H}_s}Cov(X^{\overline{H}}_t,X^{\overline{H}}_s).$ %\qed
\end{proof}

\begin{remark} If $H$ is constant, Proposition \ref{gss} yields the usual self-similarity property of \emph{fBM}.
\end{remark}

\subsection{Continuity of the paths}

The path regularity of \emph{fBMvH} is closely related to the regularity properties of the parameter function $H.$ In the following paragraphs, it is shown that a continuous $H$ is a prerequisite for any kind of continuity of the sample paths and that discontinuities of $H$ yield discontinuities in the paths of $X^H$. If $H$ is H\"older continuous, the process possesses an almost surely continuous modification and is H\"older continuous as well. Furthermore, $H$ then impacts the path regularity directly at each point $t$, i.e., the local H\"older exponent of \emph{fBMvH} at time $t$ is given by $H_t.$\\
 
\noindent\underline{Continuity in probability and discontinuities}\\

\noindent It will be useful to establish that a continuous $H$ implies a continuous covariance function $\RH$. This fact can then be used to link continuity of the parameter function to stochastic continuity of the process $X^H$.

\begin{lemma}\label{contH} The function $H:[0,\infty)\rightarrow(1/2,1)$ is continuous on $\mathbb{R}_{+}$ if and only if the covariance function $\RH(\cdot,\cdot)$ is continuous in each variable and the variance $Var(X^H_{\cdot})$ is continuous as well.
\end{lemma}
\begin{proof} First, assume that $H_t$ is continuous for $t\in\mathbb{R}_{+}$. Then, $Var(X^H_{t})=t^{2H_t}$ is clearly continuous in $t$ as well. To establish continuity of $\RH(\cdot,\cdot)$ consider the following cases.

\noindent\underline{Case 1:} Let $t$ be fixed and $(s_n)$ be a sequence converging to $s>0$ with $t>s$. For all $n$ sufficiently large $s_n<t$. The convergence of the integral terms in form (\ref{form2}) of $\RH(t,s_n)$ to the integral terms of $\RH(t,s)$ follows by Lebesgues's dominated convergence theorem. The non-integral term is clearly continuous.\medskip\\ 
\underline{Case 2:} For a sequence $(s_n)$ with $s_n\downarrow0$ and $t>0$ fixed, $\RH(t,s_n)\rightarrow0$, since by Lebesgues's dominated convergence theorem, the integral part of form (\ref{form1.1}) converges to 0. If $t=0$, convergence of $\RH(s_n,0)$ to 0 becomes obvious.\medskip\\
\underline{Case 3:} If $s$ is fixed and $(t_n)$ is a sequence converging to $t>s>0$, the continuity of $\RH(\cdot,s)$ at $t$ can be established using the same arguments as in \\Case 1.\medskip\\
\underline{Case 4:} To see that $\RH(t_n,t)\rightarrow\RH(t,t)$ for $t_n\rightarrow t$, one can first consider increasing and decreasing sequences and proceed as in Case 1 and then  use $\RH(t_n,t)=\RH(t_n,t)\mathbbm{1}_{\{t_n\leq t\}}+\RH(t_n,t)\mathbbm{1}_{\{t_n>t\}}$ for an arbitrary sequence converging to $t$.\\

\noindent Now suppose that $\RH(\cdot,\cdot)$ is continuous in each variable and that $Var(X_{\cdot}^H)$ is continuous as well. $H$ has to be continuous at any $t_0\in\mathbb{R}_+-\{1\}$, since $t^{2H_t}\rightarrow t_0^{2H_{t_0}}$ for $t\rightarrow t_0$. That $H$ is also continuous at $t_0=1$ can be deduced from the form of the coefficient functions in (\ref{form2}). %\qed
\end{proof}
 
\begin{proposition}\label{cp} The process $X^H$ is continuous in probability if and only if $H$ is continuous on $\mathbb{R}_+.$
\end{proposition}

\begin{proof} By Theorem 8.12 in \cite{Janson}, $X^H$ is stochastically continuous if and only if $\RH(\cdot,\cdot)$ is continuous in each variable and $Var(X_{\cdot}^H)$ is continuous as well. The claim follows by Lemma \ref{contH}. %\qed
\end{proof}

The following Lemma will be useful for showing that a discontinuity in $H$ at $t_0$ implies either a jump of $X^H$ at $t_0$ or a sequence of discontinuity points converging to $t_0.$

\begin{lemma}\label{discH} If $H$ is discontinuous at $t_0>0$, then\medskip\\
i) $\RH(\cdot,t_0)$ is discontinuous at $t_0$; \medskip\\
ii) $X^H$ is stochastically discontinuous at $t_0.$
\end{lemma}
\begin{proof} Part (i): Let $(t_n)$ be a sequence such that $t_n\rightarrow t_0$ and $H_{t_n}\rightarrow\hat{H}_{t_0}\neq H_{t_0}.$ Using the same reasoning as in the proof of Lemma \ref{contH}, one can show that $\RH(t_n,t_0)\rightarrow\frac{c_{t,\hat{t}_0}}{H_{t_0}+\hat{H}_{t_0}}\tilde{\beta}_{t_0,\hat{t}_0}t_0^{H_{t_0}+\hat{H}_{t_0}},$ where the subscript $\hat{t}_0$ indicates that $H_{t_0}$ has to be replaced by $\hat{H}_{t_0}$ in the coefficient function in question.\\

\noindent Part(ii): Let $(t_n)$ be again a sequence such that $t_n\rightarrow t_0$ and $H_{t_n}\rightarrow\hat{H}_{t_0}\neq H_{t_0}$ and $\varepsilon>0.$ Then
\begin{eqnarray*}
P(|X^H_{t_n}-X^H_{t_0}|>\varepsilon)&=&P\left(\frac{|X^H_{t_n}-X^H_{t_0}|}{\sqrt{Var(X^H_{t_n}-X^H_{t_0})}}>\frac{\varepsilon}{\sqrt{Var(X^H_{t_n}-X^H_{t_0})}}\right)\bigskip\\
&=&2\left(1-\Phi(\frac{\varepsilon}{\sqrt{Var(X^H_{t_n}-X^H_{t_0})}})\right),
\end{eqnarray*}
where $\Phi$ denotes the Gaussian cumulative distribution function. Due to part (i) and its proof $Var(X^H_{t_n}-X^H_{t_0})=t_n^{2H_{t_n}}+t_0^{2H_{t_0}}-2\RH(t_n,t_0)$ converges to some constant $c>0$ depending on $t_0$ and the sequence $(t_n)$. Consequently, $P(|X^H_{t_n}-X^H_{t_0}|>\varepsilon)\nrightarrow0.$ %\qed
\end{proof}

%Recall that a stochastic process $(X_t)_{t\in\mathcal{T}}$ with $\mathcal{T}\subseteq\mathbb{R}$ is called separable with respect to a system of Borel sets $\mathcal{A}$ if there exists a countable set of values $(t_j)$ in $\mathcal{T}$ and a set $\Lambda$ of measure zero such that for any $A\in\mathcal{A}$ and any open interval $I,$
%$$\{X_{t_j}\in A \mbox{ for all }t_j\in I\cap\mathcal{T}\}\backslash\{X_{t}\in A \mbox{ for all }t\in I\cap\mathcal{T}\}\subseteq\Lambda.$$
%This essentially means, that the behaviour of the paths of $X$ is determined by a countable set $(X_{t_j})$. $X$ is said to be separable, if $\mathcal{A}$ is the set of all closed intervals in $\mathbb{R}.$ It is well known, that any process indexed by a subset of $\mathbb{R}$ has a separable version. (See for example Theorem 2.4 in \cite{doob}). Also, any continuous version is automatically separable.\\

 For the next proposition, consider $X^H$ to be separable. It is well known that any process indexed by a subset of $\mathbb{R}$ has a separable version (see for example Theorem 2.4 in \cite{doob}). Also, any continuous version is automatically separable.

\begin{proposition} If $H$ has a discontinuity at $t_0\neq 0$, then $X^H$ has almost surely discontinuous sample paths. Then,$P(X^H\mbox{ is discontinuous at }t_0)>0$ and the paths of $X^H$ have either a discontinuity at $t_0$ or a sequence of discontinuity points converging to $t_0$.
\end{proposition}
\begin{proof} By Lemma \ref{discH}, part (ii) it follows that $X^H$ is discontinuous in probability at $t_0.$ Hence, the process is also not almost surely continuous at $t_0$ and\\
$P(X^H\mbox{ is discontinuous at }t_0)>0$. Let\\
$A_n=\{X^H\mbox{ is discontinuous on }[t_0-1/n,t_0+1/n]\}.$ It follows that 
$$P(A_n)\geq P(X^H\mbox{ is discontinuous at }t_0)>0.$$
Theorem 2 in \cite{camba} states that a separable Gaussian process is continuous on a closed interval with probability either 0 or 1, which yields that $P(A_n)=1$ for all $n$. Since $A_n\downarrow A=\bigcap_{n\geq1}A_n$, by continuity $$P(A)=\lim\limits_{n\rightarrow\infty}P(A_n)=1.$$ 
%\qed
\end{proof}

\noindent\underline{Existence of an a.s. continuous and H\"older continuous modification}\\

\noindent In \cite{multifrac} it was shown that for an $\alpha$-H\"older continuous parameter functions $H$ the modification given in (\ref{BMrep}) has a. s. continuous sample paths which are also a.s. H\"older continuous on any compact interval $[a,b]\in\mathbb{R}_+$ with exponents in $(0,\alpha\wedge\min_{t\in[a,b]}H_t)$. The authors used the particular form of the process as well as the Garsia-Rodemich-Rumsey inequality to establish the result. In \cite{bouf}, representation (\ref{BMrep}) is shown to be a. s. continuous under the assumptions that $H$ has values in a compact interval and that the H\"older parameter of $H$ satisfies $\alpha>\max H_t$. If $H$ assumes values in a compact interval and is $\alpha_T-$H\"older continuous on each interval $[0,T]$, a concise proof for the existence of a modification of $X^H$, which is H\"older continuous on each $[0,T]$ with exponents in $(0,\frac{\alpha_T}{2})$, can be found in \cite{diss}. This proof utilizes a slightly modified version of Kolmogorov's continuity theorem.\\

At this point, it remains unclear whether an a.s. continuous modification still exists if the H\"older continuity condition on $H$ is relaxed to continuity. A necessary condition for the existence of an a.s. continuous modification can be formulated in terms of the existence of a majorizing measure (see for example Theorem 12.9 in \cite{Ledoux}).

\subsection{Asymptotic behavior of $\RH$ and local H\"older continuity}

\noindent In the following lemmas and theorem, the asymptotic behavior of $\RH(t+h,t)$ and $Var(X^H_{t+h}-X^H_t)$ will be analyzed for $h\rightarrow0$ and $\RH(t+h,t)$ will be studied in the case $h\rightarrow\infty$ as well. A direct application is a local H\"older continuity result. 

\begin{lemma}\label{ah0} For any $t>0$ fixed and $h\downarrow0$,

\begin{IEEEeqnarray*}{lLl}
&&R_H^T(t+h,t)=\frac{c_{t+h,t}}{H_{t+h}+H_t}\Big[\beta_{t+h,t}^1\beta_{t+h,t}^2t^{H_{t+h}+H_t}+\beta_{t,t+h}^1\beta_{t,t+h}^2(t+h)^{H_{t+h}+H_t}\bigskip\\
&&\hspace{2.5cm}-\frac{\beta_{t,t+h}^1}{H_{t+h}+H_t-1}\left(\frac{t+h}{t}\right)^{H_{t+h}-H_t+1}h^{H_{t+h}+H_t}+O(h^{H_{t+h}+H_t+1})\Big].
\end{IEEEeqnarray*}

\noindent For $h\uparrow0,$
\begin{IEEEeqnarray*}{lLl}
&&R_H^T(t,t+h)=\frac{c_{t,t+h}}{H_{t+h}+H_t}\Big[\beta_{t,t+h}^1\beta_{t,t+h}^2(t+h)^{H_{t+h}+H_t}+\beta_{t+h,t}^1\beta_{t+h,t}^2t^{H_{t+h}+H_t}\hspace{2cm}\bigskip\\
&&\hspace{2.5cm}-\frac{\beta_{t+h,t}^1}{H_{t+h}+H_t-1}\left(\frac{t}{t+h}\right)^{H_{t}-H_{t+h}+1}|h|^{H_{t+h}+H_t}+O(|h|^{H_{t+h}+H_t+1})\Big].
\end{IEEEeqnarray*}
\end{lemma}

\begin{proof} First, assume $t>h>0$ for all $h$ small enough. By (\ref{2F1}) 
\begin{IEEEeqnarray*}{lLl}
\,_2F_1(1,2H_{t+h};H_t+H_{t+h};-\frac{h}{t})-\,_2F_1(1,H_{t+h}-H_t;H_t+H_{t+h};-\frac{h}{t})&&\bigskip\\
\hspace{2.5cm}=-\frac{h}{t}+\sum\limits_{n=0}^\infty\frac{(2H_{t+h})_{n+2}-(H_{t+h}-H_t)_{n+2}}{(H_{t+h}+H_t)_{n+2}}\left(-\frac{h}{t}\right)^{n+2}&&
\end{IEEEeqnarray*}
and since $(x)_{n+k}=(x)_k(x+k)_n,$
\begin{IEEEeqnarray*}{lLl}
\left|\sum\limits_{n=0}^\infty\frac{(2H_{t+h})_{n+2}-(H_{t+h}-H_t)_{n+2}}{(H_{t+h}+H_t)_{n+2}}\left(-\frac{h}{t}\right)^{n+2}\right|&&\bigskip\\
\hspace{2cm}\leq\left(\frac{h}{t}\right)^2\Bigg[\frac{(2H_{t+h})_2}{(H_{t+h}+H_t)_2}\sum\limits_{n=0}^\infty\frac{(2H_{t+h}+2)_n}{(H_{t+h}+H_t+2)_n}\left(\frac{h}{t}\right)^n&&\bigskip\\
\hspace{3.5cm}+\frac{|(H_{t+h}-H_t)_2|}{(H_{t+h}+H_t)_2}\sum\limits_{n=0}^\infty\frac{(H_{t+h}-H_t+2)_n}{(H_{t+h}+H_t+2)_n}\left(\frac{h}{t}\right)^n\Bigg]&&\bigskip\\
\hspace{2cm}\leq C\left(\frac{h}{t}\right)^2\left[\,_2F_1(1,4;3;\frac{h}{t})+\,_2F_1(1,2.5;3;\frac{h}{t})\right]&&.
\end{IEEEeqnarray*}
for some $C>0.$ The Gauss hypergeometric series (\ref{2F1}) has 1 as radius of convergence and hence  $\,_2F_1(1,4;3;\frac{h}{t})+\,_2F_1(1,2.5;3;\frac{h}{t})\rightarrow2$ for $h\rightarrow0$. Plugging the above into (\ref{form3}) and \medskip noting that $\btho$ is bounded since $H_t\subset(1/2,1)$ is fixed yields the first part of the claim. The second part is obtained analogously. %\qed
\end{proof}

In the following, let $\vth=\Vtth$ and for $h>0$ let 
\begin{IEEEeqnarray*}{lLl}
&&\widetilde{\sigma}_{t,h}^2=t^{2H_t}+(t+h)^{2H_{t+h}}-\frac{2c_{t+h,t}}{H_{t+h}+H_t}\Bigg[\beta_{t+h,t}^1\beta_{t+h,t}^2t^{H_{t+h}+H_t}+\beta_{t,t+h}^1\beta_{t,t+h}^2(t+h)^{H_{t+h}+H_t}\bigskip\\
&&\hspace{2cm} -\frac{\beta_{t,t+h}^1}{H_{t+h}+H_t-1}\left(\frac{t+h}{t}\right)^{H_{t+h}-H_t+1}h^{H_{t+h}+H_t}\Bigg].
\end{IEEEeqnarray*}
In this definition, the covariance part of $\vth$ was replaced by the asymptotically equivalent function found in Lemma \ref{ah0}.
\begin{theorem}\label{avar} Let $H$ be H\"older continuous with exponent $1\geq\alpha>\sup\limits_tH_t$ and $H\in[a,b]\subset(1/2,1).$ For any fixed $t$ and $h\rightarrow0$,
$$\frac{\vth}{|h|^{2H_t}}\rightarrow1.$$
\end{theorem}
\begin{proof} First, assume $h>0$. If $t=0,$ then $\frac{\sigma_{0,h}^2}{|h|^{2H_0}}=h^{2H_h-2H_0}\rightarrow1$, hence let $t>0$.\\

\noindent\underline{Step 1}: $\widetilde{\sigma}_{t,h}^2/|h|^{2H_t}\rightarrow1$ for $h\rightarrow0.$\medskip\\
The coefficient functions 
$$f(H_{t+h},H_t)=\frac{2c_{t+h,t}\beta_{t+h,t}^1\beta_{t+h,t}^2}{H_{t+h}+H_t}\mbox{ and }f(H_t,H_{t+h})=\frac{2c_{t+h,t}\beta_{t,t+h}^1\beta_{t,t+h}^2}{H_{t+h}+H_t}$$
are both smooth with bounded partial derivatives of any order on $[a,b]\times[a,b].$ Denoting $\partial_xf(x,y)|_{x=y}$ by $g(y)$, it follows by Taylor's theorem that
\begin{eqnarray*}
f(H_{t+h},H_t)&=&1+g(H_t)(H_{t+h}-H_t)+ O((H_{t+h}-H_t)^2)\mbox{ and}\bigskip\\
f(H_t,H_{t+h})&=&1+g(H_{t+h})(H_t-H_{t+h})+ O((H_{t+h}-H_t)^2)\bigskip\\
&=&1+g(H_t)(H_t-H_{t+h})+ O((H_{t+h}-H_t)^2).
\end{eqnarray*}
By the H\"older continuity of $H$ it follows that 

\begin{IEEEeqnarray*}{lLl}
&&\widetilde{\sigma}_{t,h}^2=t^{2H_t}-t^{H_{t+h}+H_t}+(t+h)^{2H_{t+h}}-(t+h)^{H_{t+h}+H_t}\bigskip\\
&&\ \ \ \ \ \ \ \ \ \ \ \ \ -g(H_t)(H_{t+h}-H_t)(t^{H_{t+h}+H_t}-(t+h)^{H_{t+h}+H_t})\bigskip\\
&&\ \ \ \ \ \ \ \ \ \ \ \ \ \ \ \ \ \  +\frac{2c_{t+h,t}}{(H_{t+h}+H_t)}\frac{\beta_{t,t+h}^1}{(H_{t+h}+H_t-1)}\left(\frac{t+h}{t}\right)^{1+H_{t+h}-H_t}h^{H_{t+h}+H_t}+O(h^{2\alpha}).
\end{IEEEeqnarray*}

\noindent Using Taylor expansions to see 
\begin{equation*}
     \begin{aligned}
       & t^{2H_t}-t^{H_{t+h}+H_t}=\ln{(t)}t^{2H_t}(H_t-H_{t+h})+O(h^{2\alpha}),\bigskip\\
       & (t+h)^{2H_{t+h}}-(t+h)^{H_{t+h}+H_t}=-\ln{(t+h)}(t+h)^{2H_{t+h}}(H_t-H_{t+h})+O(h^{2\alpha}),\bigskip\\
       & \ln{(t+h)}=\ln{(t)}+O(h)\mbox{ and }\bigskip\\
       & t^{2H_t}-(t+h)^{2H_{t+h}}=t^{2H_t}-t^{2H_{t+h}}+O(h)=O(h^{\alpha})+O(h)\mbox{ and}\bigskip\\
       & t^{H_{t+h}+H_t}-(t+h)^{H_{t+h}+H_t}=O(h),
     \end{aligned}
\end{equation*}
one obtains
\begin{IEEEeqnarray*}{lLl}
&&\widetilde{\sigma}_{t,h}^2=\frac{2c_{t+h,t}}{(H_{t+h}+H_t)}\frac{\beta_{t,t+h}^1}{(H_{t+h}+H_t-1)}\left(\frac{t+h}{t}\right)^{1+H_{t+h}-H_t}h^{H_{t+h}+H_t}+O(h^{1+\alpha})+O(h^{2\alpha}),
\end{IEEEeqnarray*}
and hence
\begin{IEEEeqnarray*}{lLl}
\lim\limits_{h\rightarrow0}\frac{\widetilde{\sigma}_{t,h}^2}{h^{2H_t}}&=&\lim\limits_{h\rightarrow0}\frac{2c_{t+h,t}}{(H_{t+h}+H_t)}\frac{\beta_{t,t+h}^1}{(H_{t+h}+H_t-1)}\left(\frac{t+h}{t}\right)^{1+H_{t+h}-H_t}h^{H_{t+h}-H_t}=1.
\end{IEEEeqnarray*}

\noindent\underline{Step 2}: $\vth/\widetilde{\sigma}_{t,h}^2\rightarrow1.$
$$\frac{\vth}{\widetilde{\sigma}_{t,h}^2}=1+\frac{\vth-\widetilde{\sigma}_{t,h}^2}{\widetilde{\sigma}_{t,h}^2}=1+\frac{O(h^{1+H_{t+h}+H_t})}{h^{2H_t}}\left(\frac{h^{2H_t}}{\widetilde{\sigma}_{t,h}^2}\right)=1+o(1)\frac{h^{2H_t}}{\widetilde{\sigma}_{t,h}^2}\rightarrow1$$

\noindent for $h\rightarrow0$ by Lemma \ref{ah0}.\\

\noindent Combining steps 1 and 2, the statement follows for any null sequence with $h>0$ for all $h$:
$$\frac{\vth}{h^{2H_t}}=\frac{\vth}{\widetilde{\sigma}_{t,h}^2}\left(\frac{\widetilde{\sigma}_{t,h}^2}{h^{2H_t}}\right)\rightarrow1.$$

\noindent If the second part of Lemma \ref{ah0} is used to define $\widetilde{\sigma}_{t,h}^2$, then the claim follows in the case $h<0$ for all $h$ by the same arguments as above. Finally, the statement holds for any sequence $h\rightarrow0$ since
$$\left|\frac{\sigma^2_{t,h}}{|h|^{2H_t}}-1\right|\leq\left|\frac{\sigma^2_{t,|h|}}{|h|^{2H_t}}-1\right|+\left|\frac{\sigma^2_{t,-|h|}}{|h|^{2H_t}}-1\right|\rightarrow0\mbox{ for }h\rightarrow0.$$
%\qed
\end{proof}

\begin{corollary} Let $H$ be H\"older continuous with exponent $1\geq\alpha>\sup\limits_tH_t$ and $H\in[a,b]\subset(1/2,1).$ For any fixed $t$ and $h\rightarrow0$
$$\frac{X^H_{t+h}-X^H_t}{|h|^{H_t}}\stackrel{\mathcal{L}}{\longrightarrow}\mathcal{N}(0,1).$$
\end{corollary}
\begin{proof} The claim follows by L\'{e}vy's continuity theorem combined with Theorem 
\ref{avar}.
%\qed
\end{proof}

The next corollary uses Theorem \ref{avar} in order to show that under the H\"older continuity assumption on $H$, the local regularity of the path of \emph{fBMvH} at each point $t$ is governed by the function $H$. $H_t$ is identified as the local H\"older exponent for $X^H$ at time $t$. In \cite{multifrac} and \cite{bouf}, this result was obtained using the representation of the process as an integral with respect to \emph{BM} and an alternative definition of local H\"older exponent, and an asymptotic self-similarity result respectively.
 
\begin{definition}\label{lHE} A function $f$ has local H\"older exponent $\gamma$ at $t$ if
$$\gamma=\sup\,\{\alpha:\lim\limits_{h\rightarrow0}\frac{|f(t+h)-f(t)|}{|h|^\alpha}=0\}.$$
This is equivalent to the following:
$$\lim\limits_{h\rightarrow0}\frac{|f(t+h)-f(t)|}{|h|^\alpha}=0\mbox{ for all }0<\alpha<\gamma\mbox{ and }$$
$$\limsup\limits_{h\rightarrow0}\frac{|f(t+h)-f(t)|}{|h|^\alpha}=\infty\mbox{ for all }\alpha>\gamma>0.$$
\end{definition}

In the following corollary, $X^H$ is again assumed to be separable. 

\begin{corollary} Let $H$ be H\"older continuous with exponent $1\geq\alpha>\sup\limits_tH_t$ and $H\in[a,b]\subset(1/2,1).$ Then, for each $t>0$ the local H\"older exponent of $X^H$ at $t$ is almost surely $H_t$.
\end{corollary}
\begin{proof} 
Let $t\geq0$ be fixed and $0<\gamma<H_t.$ Further, let $(h_m)_{m\geq1}$ be a null sequence and 
$$E=\left\{\limsup\limits_{m\rightarrow\infty}\frac{|X^H_{t+h_m}-X^H_t|}{|h_m|^\gamma}>0\right\}=\bigcup\limits_{n=1}^\infty A_n,$$
where 
$$A_n=\limsup\limits_{m\rightarrow\infty}B^n_m,\ B^n_m=\left\{\frac{|X^H_{t+h_m}-X^H_t|}{|h_m|^\gamma}>\frac{1}{n}\right\}.$$
In order to show $P(E)=0$, $P(A_n)=0$ for all $n$ will be established first. Let $Z\sim\mathcal{N}(0,1)$ and let $\Phi$ again denote the cumulative distribution function of a standard normal. It follows that 
\begin{eqnarray*}
P(B^n_m)&=&1-P\left(-\frac{|h_m|^\gamma}{n}\leq X^H_{t+h_m}-X^H_t\leq\frac{|h_m|^\gamma}{n}\right)\bigskip\\
&=&1-P\left(-\frac{|h_m|^\gamma}{n\sigma_{t,h_m}}\leq Z\leq\frac{|h_m|^\gamma}{n\sigma_{t,h_m}}\right)\bigskip\\
&=&2\left(1-\Phi(\frac{|h_m|^\gamma}{n\sigma_{t,h_m}})\right)\bigskip\\
&\leq&2\frac{n\sigma_{t,h_m}}{|h_m|^\gamma}\frac{1}{\sqrt{2\pi}}e^{-\frac{1}{2}(|h_m|^\gamma/n\sigma_{t,h_m})^2}
\end{eqnarray*} 
by the known inequality $1-\Phi(x)\leq\frac{1}{\sqrt{2\pi}}\frac{1}{x}e^{-\frac{x^2}{2}}$ for $x>0$.\\
By Theorem \ref{avar} 
$$\sum\limits_{m=1}^\infty P(B^n_m)<\infty$$ and hence that $P(A_n)=0$ by the Borel-Cantelli lemma. Consequently, $P(E)=0.$ Thus,
$$\limsup\limits_{m\rightarrow\infty}\frac{|X^H_{t+h_m}-X^H_t|}{|h_m|^\gamma}=0\ \mbox{a.s.}$$
Since the process is separable and the sequence $(h_m)_{m\geq1}$ was chosen arbitrarily, it follows that $\limsup\limits_{h\rightarrow0}\frac{|X^H_{t+h}-X^H_t|}{|h|^\gamma}=0$ and hence
$$\lim\limits_{h\rightarrow0}\frac{|X^H_{t+h}-X^H_t|}{|h|^\gamma}=0\mbox{ for }\gamma<H_t.$$

\noindent Now, let $\gamma>H_t$ and  $(h_n)_{n\geq1}$ be a null sequence. Then, $$\frac{|h_n|^\gamma}{|X^H_{t+h_n}-X^H_t|}\stackrel{P}{\longrightarrow}0\mbox{ for }n\rightarrow\infty,$$
\vspace{.1cm}

\noindent since for any fixed $\varepsilon>0$

\begin{eqnarray*}
P\left(\frac{|h_n|^\gamma}{|X^H_{t+h_n}-X^H_t|}>\varepsilon\right)&=&P\left(-\frac{|h_n|^\gamma}{\varepsilon}\leq X^H_{t+h_m}-X^H_t\leq\frac{|h_n|^\gamma}{\varepsilon}\right)\bigskip\\
&=&P\left(-\frac{|h_n|^\gamma}{\varepsilon\sigma_{t,h_n}}\leq Z\leq\frac{|h_n|^\gamma}{\varepsilon\sigma_{t,h_n}}\right)\bigskip\\
&=&2\Phi\left(\frac{|h_n|^\gamma}{\varepsilon\sigma_{t,h_n}}\right)-1\rightarrow0 \mbox{ for }n\rightarrow\infty
\end{eqnarray*}
by Theorem \ref{avar}. Consequently, there exists a subsequence $(h_{n_k})_{k\geq1}$ for which
$$\frac{|h_{n_k}|^\gamma}{|X^H_{t+h_{n_k}}-X^H_t|}\stackrel{k\rightarrow\infty}{\longrightarrow}0\mbox{ a.s.}$$
and hence
$$\limsup\limits_{n\rightarrow\infty}\frac{|X^H_{t+h_n}-X^H_t|}{|h_n|^\gamma}\geq\limsup\limits_{k\rightarrow\infty}\frac{|X^H_{t+h_{n_k}}-X^H_t|}{|h_{n_k}|^\gamma}=\infty.$$
%\qed
\end{proof}

The next lemma concerns the asymptotic behavior of $\RH(t+h,t)$ for $h\rightarrow\infty.$ It will be useful in establishing the long-range dependence property of \emph{fBMvH}.\\

For two real-valued functions $f_1$ and $f_2$ on $\mathbb{R}$, $f_1\sim f_2$ will denote that there exist constants $C>c>0$  such that $0<c<f_1(h)/f_2(h)<C<\infty$ for all $h$ sufficiently large.

\begin{lemma}\label{ahi} Let $H:[0,\infty)\rightarrow[a,b]\subset(1/2,1).$ For any $t$ fixed 
\begin{IEEEeqnarray}{lLl}\label{asy}
R_H(t+h,t)=g_{t,1}(h)h^{2H_{t+h}-1}+\mathbbm{1}_{\{H_t=H_{t+h}\}}g_{t,2}(h)[(t+h)^{2H_{t+h}}-(t+h)h^{2H_{t+h}-1}]+g_{t,3}(h),&&\nonumber
\end{IEEEeqnarray}
where $g_{t,1}(h)\sim1$, $g_{t,2}(h)\sim1$ and $g_{t,3}(h)\sim1$ as $h\rightarrow\infty$.
\end{lemma}

\begin{proof} For simplicity, let $\beta^*=\beta(H_{t+h}+H_t-1,1-2H_{t+h})$. Note that the Beta function is defined for non-integer negative arguments by $\beta(a,b)=\frac{\Gamma(a)\Gamma(b)}{\Gamma(a+b)},$ where $\Gamma(\cdot)$ denotes the Gamma function. $\beta^*$ is bounded since $H\in[a,b].$ Applying Lemma \ref{Abramo}, part $(vi)$, Appendix B, and using that $\,_2F_1(a,b;c;z)=1$ whenever $a=0$ or $b=0$ (see (\ref{2F1}))
\begin{IEEEeqnarray}{lll}
\,_2F_1(1,2H_{t+h};H_{t+h}+H_t;-\frac{h}{t})&& \nonumber\\
\ \ =(H_{t+h}+H_t-1)\left(\frac{t}{t+h}\right)\left[\frac{\,_2F_1(1,H_t-H_{t+h};2-2H_{t+h};\frac{t}{t+h})}{2H_{t+h}-1}\right.&&\label{eq:1}\\
\ \ \ \ \ \ +\: \left.\mathbbm{1}_{\{H_t\neq H_{t+h}\}}\beta^*\left(\frac{t}{t+h}\right)^{2H_{t+h}-1}\left(\frac{h}{t+h}\right)^{1-H_{t+h}-H_t}\right]&&\nonumber
\end{IEEEeqnarray}
and 
\begin{IEEEeqnarray}{lll}
\,_2F_1(1,H_{t+h}-H_t;H_{t+h}+H_t;-\frac{h}{t})&& \nonumber\\
\ \ =\mathbbm{1}_{\{H_t=H_{t+h}\}}+\mathbbm{1}_{\{H_t\neq H_{t+h}\}}(H_{t+h}+H_t-1)\left(\frac{t}{t+h}\right)&&\label{eq:2}\\
\ \ \ \ \ \ \times\:\left[\frac{\,_2F_1(1,2H_{t};2-H_{t+h}+H_t;\frac{t}{t+h})}{H_{t+h}-H_t-1}+\btht\left(\frac{t}{t+h}\right)^{H_{t+h}-H_t-1}\left(\frac{h}{t+h}\right)^{1-H_{t+h}-H_t}\right]&&\nonumber
\end{IEEEeqnarray}
are obtained. Let
\begin{eqnarray*}
f_t(h)&=&\frac{\,_2F_1(1,H_t-H_{t+h};2-2H_{t+h};\frac{t}{t+h})}{2H_{t+h}-1}+\mathbbm{1}_{\{H_t\neq H_{t+h}\}}\frac{\,_2F_1(1,2H_{t};2-H_{t+h}+H_t;\frac{t}{t+h})}{1+H_t-H_{t+h}}.
\end{eqnarray*}

\noindent First, plugging (\ref{eq:1}) and (\ref{eq:2}) into form (\ref{form3}) and then regrouping the terms yields
\begin{IEEEeqnarray}{lLl}
R_H(t+h,t)&=&\frac{c_{t+h,t}}{H_t+H_{t+h}}\Bigg[\btho t\left(\frac{t+h}{th}\right)^{H_{t+h}-H_t}f_t(h)h^{2H_{t+h}-1}\nonumber\\
&&\hspace{-2cm}+\mathbbm{1}_{\{H_t=H_{t+h}\}}\Big[\btho\btht(t+h)^{H_t+H_{t+h}}-\frac{\btho}{H_t+H_{t+h}-1}(t+h)h^{H_t+H_{t+h}-1}\Big]\nonumber\bigskip\\
&&+\: \left[\bhto\bhtt+\mathbbm{1}_{\{H_t\neq H_{t+h}\}}\btho\beta^*\right]t^{H_{t+h}+H_t}\Bigg].\nonumber
\end{IEEEeqnarray}
Noting that $\beta_{t,t}^2=\frac{1}{2H_t-1}$ and that\medskip\\ $\frac{c_{t+h,t}}{H_t+H_{t+h}}\btho\btht\mathbbm{1}_{\{H_t=H_{t+h}\}}= \frac{1}{2}\mathbbm{1}_{\{H_t=H_{t+h}\}}$, let
\begin{IEEEeqnarray*}{lLl}
g_{t,1}(h)&=&\frac{c_{t+h,t}}{H_{t+h}+H_t}\btho t(\frac{1}{h}+\frac{1}{t})^{H_{t+h}-H_t}f_t(h),\bigskip\\
g_{t,2}(h)&=&\frac{1}{2},\mbox{ and}\bigskip\\
g_{t,3}(h)&=&\frac{c_{t+h,t}}{H_{t+h}+H_t}[\bhto\bhtt+\mathbbm{1}_{\{H_t\neq H_{t+h}\}}\bhto\beta^*]t^{H_{t+h}+H_t},
\end{IEEEeqnarray*}
which yields (\ref{asy}). It remains to show that $g_{t,i}(h)\sim1$ for $i=1,2,3.$ This is clearly the case for $i=2$ and $i=3,$ since $a\leq H\leq b$ and all beta functions involved are bounded. For $i=3$ one approximates
$$
\sum\limits_{n=0}^\infty\frac{(1)_n}{(2.5)_n}\left(\frac{t}{t+h}\right)^n\leq\,_2F_1(1,2H_{t};2-H_{t+h}+H_t;\frac{t}{t+h})\leq\sum\limits_{n=0}^\infty\frac{(2)_n}{(1.5)_n}\left(\frac{t}{t+h}\right)^n.
$$
Both bounds are converging to 1 for $h\rightarrow\infty$ as argued in the proof of Lemma \ref{ah0}. Similarly,
\begin{IEEEeqnarray*}{lLl}
\,_2F_1(1,H_t-H_{t+h};2-2H_{t+h};\frac{t}{t+h})&&\bigskip\\
\hspace{.5cm}=1+\left(\frac{t}{t+h}\right)\left(\frac{H_t-H_{t+h}}{2-2H_{t+h}}\right)\sum\limits_{n=0}^\infty\frac{(H_t-H_{t+h}+1)_n}{(3-2H_{t+h})_n}\left(\frac{t}{t+h}\right)^n\rightarrow1&&
\end{IEEEeqnarray*}
for $h\rightarrow\infty.$ In the case $\{H_t=H_{t+h}\},$ $f_t(h)=1/(2H_t-1).$ Hence, $f_t(h)\sim1$ and $g_{t,1}(h)\sim1$, $g_{t,2}(h)\sim1$ and $g_{t,3}(h)\sim1$ for any fixed $t$. %\qed
\end{proof}

\begin{remark} For constant $H$ the proof of Lemma \ref{ahi} yields the covariance function of \emph{fBM}.
\end{remark}

\subsection{Long-range dependence}

\emph{fBMvH} $X^H$ retains the long-range dependence property, which makes \emph{fBM} with $H>1/2$ an attractive model in situations where the dependence on past events decays slowly in time. \\

There are several ways to define long-range dependence mathematically. All of them refer to slowly decaying correlations in some way (see, e.g., \cite{Cohen}). For a second-order process $X$, let $Cor_X(t,s)$ denote the correlation of $X_t$ and $X_s$. In the next proposition, it will be shown that $Cor_{X^H}(t+h,t)$ decays not faster than $h^{-1}$ for $h\rightarrow\infty$ and the definition of long-range dependence used here will be the following:
\begin{definition}\label{long} A second-order process $X$ is said to have long-range dependence if 
$$\sum\limits_{k=0}^{\infty}|Cor_X(t+\delta k,t)|=\infty\mbox{ for all }t>0\mbox{ and }\delta>0.$$
\end{definition}

\begin{remark}
The correlation rather than the covariance is used in this definition. For non-stationary processes, the correlation and covariance of increments are not multiples of each other and while the sum of the covariances may diverge, the correlation can still be decreasing fast enough in time.
\end{remark}

\begin{proposition} Let $H:[0,\infty)\rightarrow[a,b]\subset(1/2,1).$ Then, $X^H$ has long-range dependence.
\end{proposition}

\begin{proof} Taylor expansions yield\bigskip\\
$h^{2H_{t+h}-1}=(t+h)^{2H_{t+h}-1}+o(1)$ for $h\rightarrow\infty$ and\bigskip\\
$h^{2H_{t+h}}=(t+h)^{2H_{t+h}}-2H_{t+h}t(t+h)^{2H_{t+h}-1}+o(1)$ for $h\rightarrow\infty.$ \bigskip\\
By Lemma \ref{ahi}, it follows that
\begin{IEEEeqnarray*}{lLl}
R_H(t+h,t)&=&\big[g_{t,1}(h)+\mathbbm{1}_{\{H_t=H_{t+h}\}}g_{t,2}(h)t(2H_{t+h}-1)\big](t+h)^{2H_{t+h}-1}+g_{t,3}(h)+o(1)\bigskip\\
&\sim&(t+h)^{2H_{t+h}-1}\mbox{ for }h\rightarrow\infty.
\end{IEEEeqnarray*}
Hence $Cor_{X^H}(t+h,t)\sim(t+h)^{H_{t+h}-1}$ and it follows that 
$$\sum\limits_{k=0}^{\infty}|Cor_{X^H}(t+\delta k,t)|\geq c\sum\limits_{k=0}^{\infty}(t+\delta k)^{-1/2}=\infty\mbox{ for all }t>0\mbox{ and }\delta>0$$
for some constant $c>0.$ %\qed
\end{proof}
\begin{remark} 
$Cor_{X^H}(t+h,t)\sim(t+h)^{H_{t+h}-1}$ implies that for each starting point $t$, the correlation structure evolves differently.
\end{remark}

\section{A Fokker-Planck equation for time-changed \emph{fBMvH}.}

This section establishes a Fokker-Plank equation (FPE) for the densities of a time-changed \emph{fBMvH} under the differentiability assumption on the parameter function $H$. The time-change process is the inverse of a stable subordinator, which yields a fractional derivative in the FPE. FPEs involving time-fractional derivatives are used as a powerful tool in the study and modeling of anomalous diffusion processes (see for example \cite{benson4}, \cite{janczura17}, \cite{metzler25}, \cite{zaslavsky37}).\\

A stable subordinator $W^\alpha$ with index $\alpha$, is a nonnegative and strictly increasing L\'{e}vy process starting at 0 and exhibiting the self-similarity property $W^{\alpha}_{ct}\sim c^{1/\alpha}W^\alpha_t$ for all $t>0$ and any $c>0$ in the sense of finite dimensional distributions \cite{applebaum}. The inverse of a stable subordinator $W^{\alpha}$ is defined by $E_t^{\alpha}=\inf\{s:\,W^\alpha_s>t\}$ for $t\geq0$. Since $W^\alpha$ is strictly increasing, $E^\alpha$ is non-decreasing and continuous and hence a suitable time-change process \cite{Meerschaert}.\\

The Caputo-Djrbashian fractional-order derivative $D^{\alpha}_t$ of order $\alpha\in(0,1)$ is given by
$$D^{\alpha}_tg(t) = \frac{1}{\Gamma(1-\alpha)}\int\limits_0^t\frac{g'(\tau)}{(t-\tau)^{\alpha}}\,d\tau.$$
By convention, $D^{1}_t=\frac{d}{dt}.$ The fractional integration operator of order $\alpha>0$ is defined via
$$J^{\alpha}_tg(t)= \frac{1}{\Gamma(\alpha)}\int\limits_0^tg(\tau)(t-\tau)^{\alpha-1}\,d\tau.$$
The relationship between the above three operators is given by $D^{\alpha}_t=J^{1-\alpha}_t\circ\frac{d}{dt}$ (for details, see \cite{Gorenflo}).\\

The following theorem is an adaptation of Proposition 1 in \cite{Hkru} to the specific setting of \emph{fBMvH}.

\begin{theorem} Let $X^H$ be a fBMvH with differentiable Hurst parameter function $H$. The transition probabilities $p(t,x)$ of $X^H$ satisfy
$$\partial_tp(t,x)=\left(H'_t\ln(t)+\frac{H_t}{t}\right)t^{2H_t}\partial_{x}^2\,p(t,x),\  t > 0,\ x\in\mathbb{R}.$$
\end{theorem}

If $H$ is differentiable, then $Var(X_t^H)=t^{2H_t}$ is differentiable as well. $t^{2H_t}$ is Laplace transformable, since $t^{2H_t}\leq t\vee t^2.$ Thus, the conditions of Theorem 3 in \cite{Hkru} are satisfied and the following theorem is obtained.\\

Let $\tilde{g}$ denote the ($t\rightarrow s$)--Laplace transform of a function $g=g(t)$ and let $\mathcal{L}^{-1}_{s\rightarrow t}$ denote the inverse Laplace transform.

\begin{theorem} Let $X^H$ be a fBMvH with $H$ differentiable on $(0,\infty)$. Further, let $E^{\alpha}$ be the inverse of a stable subordinator $W^{\alpha}$ of index $\alpha\in(0,1)$, independent of $X^H$. Then, the transition probabilities $q(t,x)$ of the time-changed process $(X^H_{E^{\alpha}_t})_{t\geq0}$ satisfy the equivalent PDEs
$$D^{\alpha}_tq(t,x) =J^{1-\alpha}_t\Lambda^{\alpha}_{X^H,t}\partial^2_{x}q(t,x),\ t>0,\ x\in\mathbb{R}$$
and
$$\partial_tq(t,x)=\Lambda^{\alpha}_{X^H,t}\partial^2_{x}q(t,x),\ t>0,\ x\in\mathbb{R},$$
where $\Lambda^{\alpha}_{X^H,t}$ is the operator acting on $t$ given by
$$
\Lambda^{\alpha}_{X^H,t}g(t)=\frac{\alpha}{2}\mathcal{L}^{-1}_{s\rightarrow t}\left[\frac{1}{2\pi i}\int_{\mathcal{C}}(s^{\alpha}-z^{\alpha})\tilde{R}_{H}(s^{\alpha}-z^{\alpha})\tilde{g}(z)\,dz\right](t),
$$
with initial condition $q(0,x) = \delta_0(x)$.  %and with $z^{\alpha} = e^{\alpha\mathrm{\mathrm{Ln}}(z)},\ \mathrm{Ln}(z)$ being the principal value of the complex logarithmic function $\ln(z)$ with cut along the negative real axis, and $\mathcal{C}$ being a curve in the complex plane obtained via the transformation $\zeta=z^{\alpha}$ which leaves all the singularities of $\tilde{R}_{H}$ on one side.
\end{theorem}
For further details on the operator $\Lambda^{\alpha}_{X^H,t}$ see \cite{Hkru}.

\begin{remark} The proof uses the fact that given the independence of $X^H$ and $E^{\alpha}$, the relationship between the transition probabilities $p(\tau,x)$ for $\tau>0,\ x\in\mathbb{R}$ of the process and the transition probabilities of the time-changed version is given by $q(t,x)=\int\limits_0^{\infty}f_{E^{\alpha}_t}(\tau)p(\tau,x)\,d\tau$, $t>0$, with $f_{E^{\alpha}_t}$ denoting the density function of $E^{\alpha}_t$. The time change yields the occurrence of the fractional-order derivative. 
\end{remark}
\begin{remark} The choice of the time-change process can be extended to the inverse of an arbitrary mixture of independent stable subordinators (see \cite{Hkru}).
\end{remark}

\section{Modeling}

The plots in this section were generated using the Cholesky decomposition of the covariance matrix. For that, let $t_1,...,t_n\in[0,T]$ for some $T>0.$ By Theorem \ref{pd}, the covariance matrix $\Sigma$ of the vector $(X^H_{t_1},...,X^H_{t_n})$ is positive definite and can hence be decomposed as 
$$\Sigma=LL^\mathrm{T},$$
where $L$ is a lower triangular matrix.
Let $Y=(Y_1,...,Y_n)$ be a standard normal vector with the identity matrix as the covariance matrix. Then, $LY$ has $\Sigma$ as covariance matrix, since
$$Cov(LY)=E[LY(LY)^\mathrm{T}]=LE[YY^\mathrm{T}]L^\mathrm{T}=LL^\mathrm{T}=\Sigma.$$
Hence, $LY$ is a sample path of a \emph{fBMvH} with parameter function $H$ and at the times $t_1,...,t_n.$\\

Using the Cholesky decomposition in order to obtain a sample path of a Gaussian process generates an exact sample path of the process in question. The disadvantage of such an algorithm is the long computational time needed to compute the $n\times n$ covariance matrix and the Cholesky factorisation matrix $L$. For applications, an algorithm based on the convergence of \emph{fBM}s with indices $H_n\rightarrow H_ t$ such as in \cite{PeltVeh} would be preferable in practice, even though it would not generate an exact sample path.\\
\vspace{-.1cm}
\noindent\underline{Plots.}
\vspace{-.5cm}
\begin{figure}[H]
\centering
\includegraphics[width=0.8\textwidth]{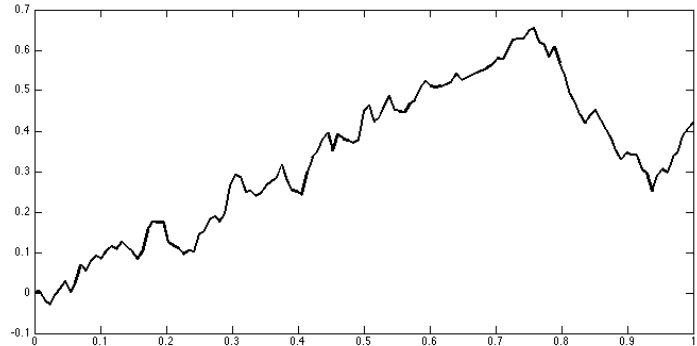}
\noindent\caption{This is the usual \emph{fBM} with constant Hurst parameter $H=.75$.}
\end{figure}
\vspace{-.6cm}
The following are exact sample paths of \emph{fBMvH} with various parameter functions $H(\cdot).$
\begin{figure}[H]
\centering
\includegraphics[width=0.8\textwidth]{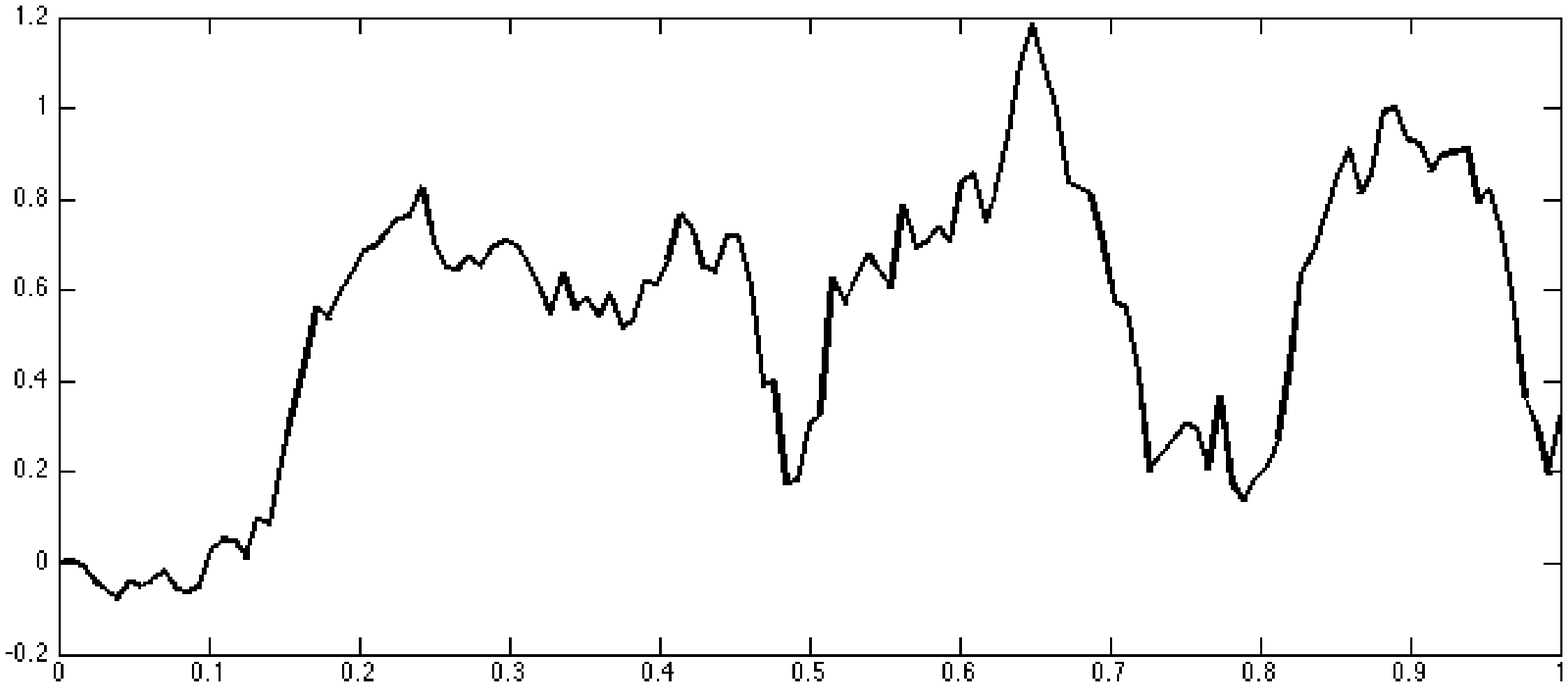}
\caption{$H(t)=.51+(t-.5)^2$ for $t<.5$ and $H(t)=.51$ for $t\geq.5.$ The Hurst parameter approaches .51 as $t\rightarrow 1/2$.}
\end{figure}
\vspace{-.7cm}
\begin{figure}[H]
\centering
\includegraphics[width=0.8\textwidth]{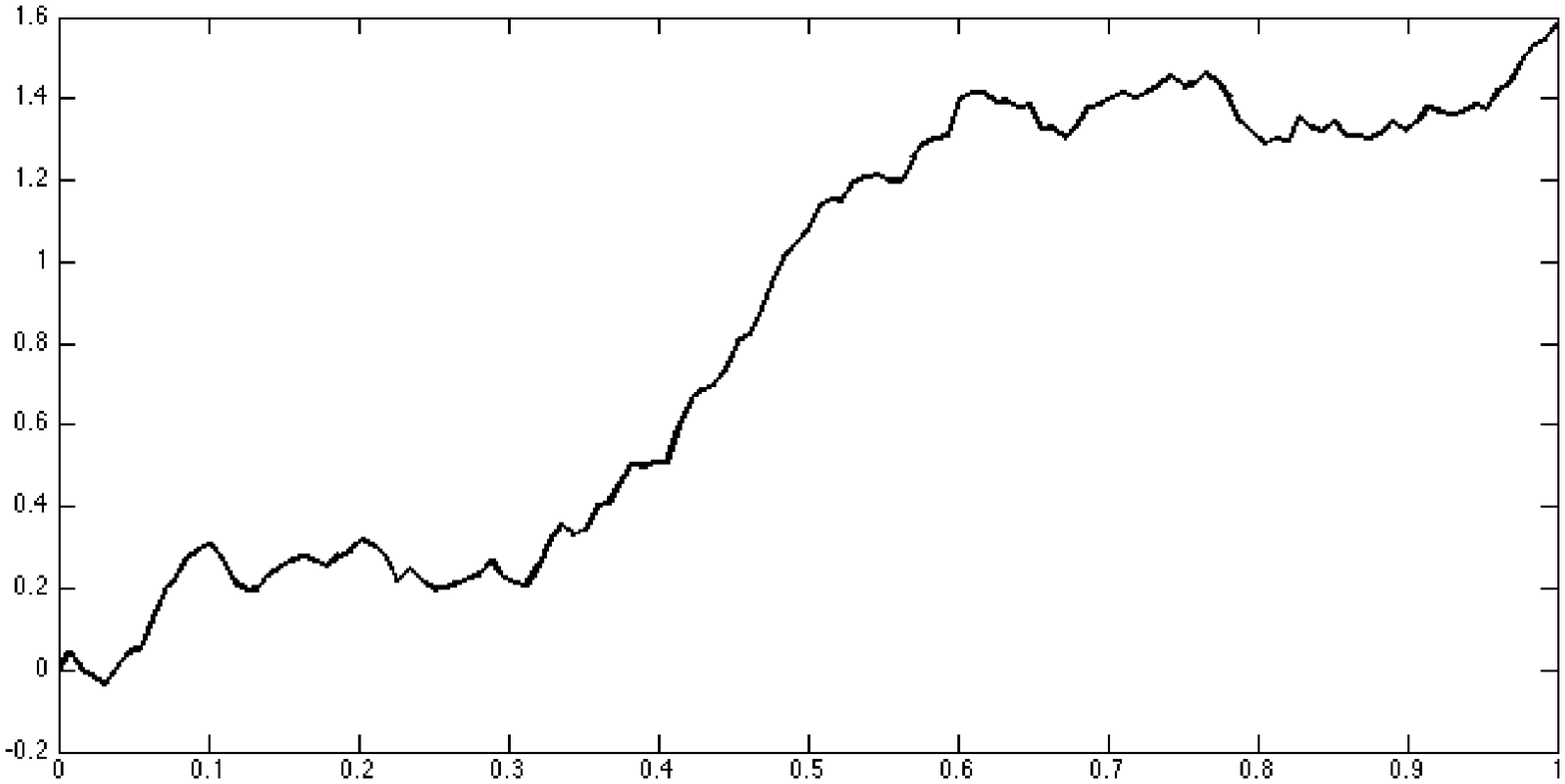}
\noindent\caption{$H(t)=\frac{2}{3}+\frac{\sin(t)}{12}$. The Hurst parameter oscillates between 7/12 and 9/12. }
\end{figure}
\begin{figure}[H]
\centering
\includegraphics[width=0.8\textwidth]{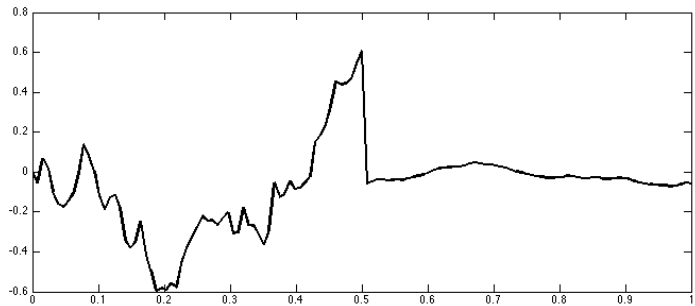}
\caption{$H(t)=.55$ for $t\leq.5$ and $H(t)=.95$ for $t>.5.$The Hurst parameter jumps at .5.}
\end{figure}
\vspace{-.7cm}
\begin{figure}[H]
\centering
\includegraphics[width=0.8\textwidth]{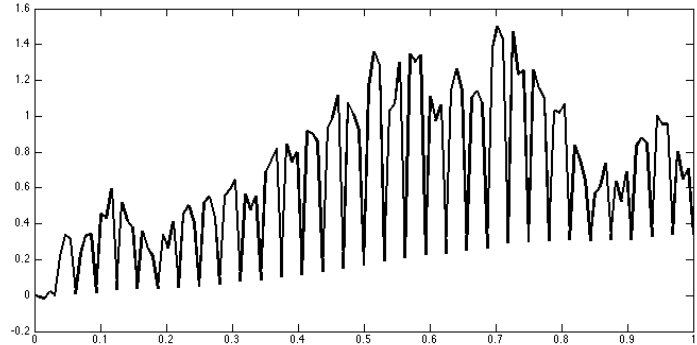}
\caption{ The Hurst parameter jumps periodically: $H(t)=.99$ if $128t\ mod\ 4 =0$ and $H(t)=.51$ else.}
\end{figure}

\noindent\textbf{Acknowledgements }The author wishes to thank Dr. Marjorie Hahn for her help and advice, Dr. Kei Kobayashi for productive comments as well as Dr. Anna P\'osfei for fruitful discussions. The author is also indebted to an anonymous referee for calling my attention to papers \cite{multifrac} and \cite{bouf} as well as suggestions which increased the clarity of the paper.

\bibliographystyle{spmpsci}
\bibliography{bib1}

\appendix
\section{Appendix}

\begin{lemma}\label{AppB1} Let $\alpha<2, \gamma>1/2,\delta>1/2$ and $I=\int_0^{a}u^{1-\alpha}(b-u)^{\gamma-\frac{3}{2}}(a-u)^{\delta-\frac{3}{2}}\,du$, then\medskip\\
\noindent i) 
	$$
	I=(b-a)^{\gamma+\delta-2}\int_{b/a}^{\infty}(x-1)^{\alpha-\delta-\gamma}x^{\gamma-\frac{3}{2}}(ax-b)^{1-\alpha}\,dx.
	$$

\noindent ii) If additionally $\alpha=\gamma+\delta$, then
	$$
	I=(b-a)^{\alpha-2}\int_{b/a}^{\infty}x^{\gamma-\frac{3}{2}}(ax-b)^{1-\alpha}\,dx=(b-a)^{\alpha-2}b^{\frac{1}{2}-\delta}a^{\frac{1}{2}-\gamma}\beta(\delta-1/2,2-\alpha).
	$$
\noindent iii) If $\alpha,\gamma\in(-1,1)$ arbitrary, $0<a<b$ and $J=\int_a^by^{-\alpha}\int_0^az^{\alpha}(y-z)^{-\gamma}\,dz\,dy,$ then
	$$
	J=\frac{a^{2-\gamma}}{2-\gamma}\int_{a/b}^1v^{\alpha+\gamma-2}(1-v)^{-\gamma}\,dv+\frac{b^{2-\gamma}}{2-\gamma}\int_0^{a/b}v^{\alpha}(1-v)^{-\gamma}\,dv-\frac{a^{2-\gamma}}{2-\gamma}\beta(\alpha+1,1-\gamma).
	$$

\end{lemma}

\begin{proof} Part i) is obtained via substituting $x=\frac{u-b}{u-a}$. Part ii) follows by substituting $y=\frac{b}{ax}.$ For iii) $z=yv$ is substituted  and it follows that
\begin{IEEEeqnarray*}{lLl}
J&=&\int_a^by^{1-\gamma}\int_0^{a/y}v^{\alpha}(1-v)^{-\gamma}\,dv\,dy\bigskip\\
&=&\int_{a/b}^1v^{\alpha}(1-v)^{-\gamma}\int_a^{a/v}y^{1-\gamma}\,dy\,dv+\int_0^{a/b}v^{\alpha}(1-v)^{-\gamma}\int_a^{b}y^{1-\gamma}\,dy\,dv\bigskip\\
&=&\frac{a^{2-\gamma}}{2-\gamma}\int_{a/b}^1v^{\alpha+\gamma-2}(1-v)^{-\gamma}\,dv+\frac{b^{2-\gamma}}{2-\gamma}\int_0^{a/b}v^{\alpha}(1-v)^{-\gamma}\,dv\bigskip\\
&&\ \ -\frac{a^{2-\gamma}}{2-\gamma}\beta(\alpha+1,1-\gamma).
\end{IEEEeqnarray*}
%\qed
\end{proof}

\section{Appendix}

Formulas $(i)-(iv)$ below can be found in \textit{The Handbook of mathematical functions}, p. 559 by Abramowitz and Stegun \cite{Abram}. Formula $(v)$ is from the \textit{Integrals and Series} handbook by Prudnikov at. al. \cite{Prud}.

\begin{lemma}\label{Abramo}
Let $a,b,c$ be real numbers and $z\in\mathbb{C}$.\medskip\\
i) For $|arg(z)|, |arg(1-z)|<\pi$ and when all terms are defined, 
\begin{IEEEeqnarray*}{lLl}
\,_2F_1(a,b;c;z)&=&\frac{\Gamma(c)\Gamma(c-a-b)}{\Gamma(c-a)\Gamma(c-b)}z^{-a}\,_2F_1(a,a-c+1;a+b-c+1;1-\frac{1}{z})\bigskip\\
&&\hspace{-.1cm}+\frac{\Gamma(c)\Gamma(a+b-c)}{\Gamma(a)\Gamma(b)}(1-z)^{c-a-b}z^{a-c}\,_2F_1(c-a,1-a;c-a-b+1;1-\frac{1}{z}).
\end{IEEEeqnarray*}
ii) If $(1-z)^{-a}$ is defined, 
$$_2F_1(a,b;c;z)=(1-z)^{-a}\,_2F_1(a,c-b;c;\frac{z}{z-1}).$$
iii) For $|arg(1-z)|<\pi$ and when all terms are defined, 
\begin{IEEEeqnarray*}{lLl}\label{abramo}
\,_2F_1(a,b;c;z)&=&\frac{\Gamma(c)\Gamma(c-a-b)}{\Gamma(c-a)\Gamma(c-b)}\,_2F_1(a,b;a+b-c+1;1-z)\bigskip\\
&&\ \ \ \ +\frac{\Gamma(c)\Gamma(a+b-c)}{\Gamma(a)\Gamma(b)}(1-z)^{c-a-b}\,_2F_1(c-a,c-b;c-a-b+1;1-z).
\end{IEEEeqnarray*}
iv) If $(1-z)^{c-a-b}$ is defined, 
$$_2F_1(a,b;c;z)=(1-z)^{c-a-b}\,_2F_1(c-a,c-b;c;z).$$
v) 
\begin{IEEEeqnarray*}{lLl}
\int_a^b(x-a)^{\alpha-1}(b-x)^{\delta-1}(cx+d)^{\gamma}\,dx&&\bigskip\\
\ \ \ \ \ \ \ \ \ \ \ \ \ \ \ =\beta(\alpha,\delta)(b-a)^{\alpha+\delta-1}(ac+d)^{\gamma}\,_2F_1(\alpha,-\gamma;\alpha+\delta;\frac{c(a-b)}{ac+d})&&
\end{IEEEeqnarray*}
if $Re(\alpha)>0,\ Re(\delta)>0$ and $|arg((d+cb)/(d+ca))|<\pi.$\medskip\\
vi) Under the assumptions of $(ii), (iii)$ and $(iv)$
\begin{IEEEeqnarray*}{lLl}
\,_2F_1(a,b;c;z)&=&\mathbbm{1}_{\{a=0\vee b=0\}}\bigskip\\
&&\hspace{-1.8cm}+\mathbbm{1}_{\{a\neq0\wedge b\neq0\}}(1-z)^{-a}\Bigg[\frac{\Gamma(c)\Gamma(b-a)}{\Gamma(c-a)\Gamma(b)}\,_2F_1(a,c-b;a-b+1;\frac{1}{1-z})\bigskip\\
&&\hspace{-1.4cm}+\frac{\Gamma(c)\Gamma(a-b)}{\Gamma(a)\Gamma(c-b)}\left(\frac{z}{z-1}\right)^{1-c}\left(\frac{1}{1-z}\right)^{b-a}\,_2F_1(b-c+1,1-a;b-a+1;\frac{1}{1-z})\Bigg].
\end{IEEEeqnarray*}
\end{lemma}
\begin{proof} Part $(vi)$: The equality is obtained by consecutively applying parts $(ii)$ and $(iii)$ of the Lemma to $\,_2F_1(a,b;c;z)$ and then applying part $(iv)$ to the second term of what was obtained in the first two steps. 
The indicator functions make up for the case  that $a=0$ or $b=0$, i.e., when $(iii)$ cannot be applied.
%\qed
\end{proof}

\end{document}